\documentclass[12pt,bezier,]{article}%draft
\usepackage{graphicx}
\usepackage{epstopdf}
\usepackage{subfigure}
\usepackage{enumerate}
\usepackage{fourier}%统一修改正文和数学字体为Adobe Utopia
\usepackage[colorlinks,anchorcolor=blue,linkcolor=blue,citecolor=blue]{hyperref}
\usepackage[all]{hypcap}%引用链接会跳转至浮动体开始的地方
\usepackage[fleqn]{amsmath}%公式左对齐
\usepackage[numbers,sort&compress]{natbib}%压缩并排序参考文献标号
\usepackage{amsthm,amsfonts,amssymb,mathrsfs,bbding}%调用定理
\usepackage{amssymb}
\usepackage{float}%图片插入指定位置
\evensidemargin=\oddsidemargin\topmargin=-1.5cm

\newtheorem{example}{Example}
\newtheorem{lem}{Lemma}[section]
\newtheorem{cor}[lem]{Corollary}
\newtheorem{thm}[lem]{Theorem}
\newtheorem{definition}[lem]{Definition}

\newtheorem{remark}{Remark}

\usepackage{CJK}

\begin{document}
\title{The normalized Laplacian spectra of subdivision vertex-edge neighbourhood  vertex(edge)-corona for graphs\footnote{This work is supported by the Young Scholars Science Foundation of Lanzhou Jiaotong  University (No.2016014) and NSFC (No.11461038).}}

\author{{\small Fei Wen$^{1,2}$, \ \ You Zhang$^{2}$, \ \ Wei Wang$^{1,}$\footnote{E-mail addresses: wang\_weiw@163.com(W. Wang), wenfei@mail.lzjtu.cn(F. Wen).}}\\[2mm]
\scriptsize 1. School of Mathematics and Statistics, Xi'an Jiaotong University, Xi'an 710049, P.R. China\\
\scriptsize 2. Institute of Applied Mathematics,
Lanzhou Jiaotong University, Lanzhou 730070, P.R.China}
\date{}
\maketitle
{\flushleft\large\bf Abstract.} In this paper, we introduce two new
graph operations, namely, \emph{the subdivision vertex-edge neighbourhood
vertex-corona} and \emph{the subdivision vertex-edge neighbourhood
edge-corona} on graphs $G_1$, $G_2$ and $G_3$, and the resulting graphs are denoted by $G_1^S\bowtie (G_2^V\cup G_3^E)$ and
$G_1^S\diamondsuit(G_2^V\cup G_3^E)$, respectively. Whereafter, the normalized Laplacian spectra of $G_1^S\bowtie (G_2^V\cup G_3^E)$ and $G_1^S\diamondsuit(G_2^V\cup G_3^E)$ are respectively determined in terms of the corresponding normalized Laplacian spectra of the connected regular graphs $G_{1}$, $G_{2}$ and $G_{3}$, which extend the corresponding results of [A. Das, P. Panigrahi, \emph{Linear Multil. Algebra}, 2017, 65(5): 962-972]. As applications, these results enable us to construct infinitely many pairs of \emph{normalized Laplacian cospectral graphs}. Moreover, we also give the number of the \emph{spanning trees}, the \emph{multiplicative degree-Kirchhoff index} and \emph{Kemeny's constant} of $G_1^S\bowtie (G_2^V\cup G_3^E)$ (resp. $G_1^S\diamondsuit(G_2^V\cup G_3^E)$).


\vspace{1mm} \maketitle {\flushleft\large\bf AMS Classifications}:
05C50

\section{Introduction}

Throughout this paper, we are concerned only with simple undirected
graphs (loops and multiple edges are not allowed). Let $G$ be a
graph with vertex set $V(G)=\{v_1,v_2,\ldots,v_n\}$ and edge set
$E(G)=\{e_1,e_2,\ldots,e_m\}$ where $|V(G)|=n$, $|E(G)|=m$. The \emph{line graph} $\ell(G)$ of graph $G$ is a graph whose
vertices corresponding the edges of $G$, and where two vertices are
adjacent iff the corresponding edges of $G$ are adjacent. We denote the \emph{complete graph} and the \emph{cycle} of order $n$ by $K_{n}$ and $C_{n}(n\ge3)$, respectively. A
\emph{graph matrix} $M=M(G)$ is defined to be a symmetric matrix
with respect to adjacency matrix $A(G)$ of $G$. The
$M$-characteristic polynomial of $G$ is defined as
$\Phi_{M}(x)=\det(xI-M)$, where $I$ is the identity matrix. The
$M$-eigenvalues of $G$ are the eigenvalues of its $M$-polynomial.
The $M$-spectrum, denoted by $Spec_{M}(G)$, of $G$ is a multiset
consisting of the $M$-eigenvalues. And two graphs $G$ and $H$ are $M$--cospectral if $\Phi_{M(G)}(x)=\Phi_{M(H)}(x)$.

Let $D(G)=diag(d_{v_1},d_{v_2},\ldots,d_{v_n})$ be the degree diagonal matrix of $G$. The graph matrix
$M=M(G)$ is respectively called the \emph{adjacency matrix},
the \emph{signless Laplacian matrix} and the \emph{normalized Laplacian matrix} of $G$ if $M$ equals $A(G)$, $Q(G)=D(G)+A(G)$ and $\mathcal{L}(G)=D(G)^{-1/2}(D(G)-A(G))D(G)^{-1/2}=I-D(G)^{-1/2}A(G)D(G)^{-1/2}$.
Conventionally, the \emph{adjacency eigenvalues} and \emph{normalized Laplacian eigenvalues} of graph $G$ are ordered respectively in non-increased sequence as follows:
$\nu_{1}\geq \nu_{2}\geq \cdots\geq \nu_{n}$ and
$\lambda_{1}\geq \lambda_{2}\geq \cdots\geq \lambda_{n}=0$. In fact, Chung \cite{Chung} has proved that $\lambda_{i}\le 2$ for all $i$.

As far as we know, many graph operations such as the \emph{disjoint union}, the \emph{corona}, the \emph{edge corona} and the \emph{neighborhood corona} have been introduced, and their spectra are computed in \cite{1,4,6,7,11,12},
respectively. The \emph{subdivision graph} $S(G)$ of a graph $G$ is the graph obtained by inserting a new vertex into every edge of $G$.
Based on subdivision graph, the graph operations  of
\emph{subdivision-vertex join} and \emph{subdivision-edge join} are
introduced in \cite{8},  and their $A$-spectrum are investigated.
Further works on their $L$-spectrum are given in \cite{9}. In
\cite{10}, the spectra of the so called \emph{subdivision-vertex
corona} and \emph{subdivision-edge corona} are computed,
respectively. In \cite{13}, the spectra of the
\emph{subdivision-vertex} and \emph{subdivision-edge neighborhood
corona} are computed, respectively. Subsequently, Song and Huang~\cite{song} obtained the $A$-spectra and $L$-spectra of the
\emph{subdivision vertex-edge corona} $G_{1}^S\circ(G_{2}^V\cup G_{3}^E)$, see $P_{4}^S\circ(P_{2}^V\cup P_{1}^E)$ for instance, which is shown in Fig.\ref{f-1}.

Motivated by the above works, we define two new graph operations
based on subdivision graph as follows. For a graph $G_1$, let
$S(G_1)$ be the subdividing graph of $G_1$ whose vertex set has two
parts: one the original vertices $V(G_1)$, another, denoted by
$I(G_1)$, the inserting vertices corresponding to the edges of
$G_1$. Let $G_2$ and $G_3$ be other two disjoint graphs.
\begin{definition}\label{d-1}
\emph{Subdivision vertex-edge neighbourhood vertex-corona}
(short for SVEV-corona) of $G_1$ with $G_2$ and $G_3$, denoted by
$G_1^S\bowtie(G_2^V\cup G_3^E)$, is the graph consisting of
$S(G_1)$, $|V(G_1)|$ copies of $G_2$ and $|I(G_1)|$ copies of $G_3$,
all vertex-disjoint, and joining the neighbours of the \emph{i}-th
vertex of $V(G_{1})$ to every vertex in the \emph{i}-th copy of
$G_{2}$ and \emph{i}-th vertex of $I(G_1)$ to each vertex in the
\emph{i}-th copy of $G_3$.
\end{definition}
For instance, we depict $P_4^S\bowtie(P_2^V\cup P_1^E)$ in
Fig.\ref{f-1}. By the definition, $G_1^S\bowtie(G_2^V\cup G_3^E)$
has $n=n_1+m_1+n_1n_2+m_1n_3$ vertices and
$m=2m_1+n_1m_2+m_1m_3+2m_1n_2+m_1n_3$ edges, where $n_i$ and $m_i$
are the number of vertices and edges of $G_i$ for $i=1,2,3$. We see
that $G_1^S\bowtie(G_2^V\cup G_3^E)$ will be a
\emph{subdivision-vertex neighbourhood corona} (see \cite{13}) if
$G_3$ is null, and will be a \emph{subdivision-edge corona} (see
\cite{10}) if $G_2$ is null. Thus \emph{subdivision vertex-edge
neighbourhood vertex-corona} can be viewed as the generalizations of
both \emph{subdivision-vertex neighbourhood corona} (denoted by
$G_1\boxdot G_2$) and \emph{subdivision-edge corona} (denoted by
$G_1\circleddash G_3$).
\begin{definition}\label{d-2}
\emph{Subdivision vertex-edge neighbourhood edge-corona} (short
for SVEE-corona) of $G_1$ with $G_2$ and $G_3$, denoted by
$G_1^S\diamondsuit(G_2^V\cup G_3^E)$, is the graph consisting of
$S(G_1)$, $|V(G_1)|$ copies of $G_2$ and $|I(G_1)|$ copies of $G_3$,
all vertex-disjoint, joining the neighbours of the \emph{i}-th
vertex of $I(G_{1})$ to every vertex in the \emph{i}-th copy of
$G_{2}$ and \emph{i}-th vertex of $V(G_1)$ to each vertex in the
\emph{i}-th copy of $G_3$.
\end{definition}

For instance, we depict $P_4^S\diamondsuit(P_2^V\cup P_1^E)$ in
Fig.\ref{f-1}. By the definition, $G_1^S\diamondsuit(G_2^V\cup
G_3^E)$ has $n=n_1+m_1+m_1n_2+n_1n_3$ vertices and
$m=2m_1+m_1m_2+n_1m_3+n_1n_3+2m_1n_2$ edges, where $n_i$ and $m_i$
are the number of vertices and edges of $G_i$ for $i=1,2,3$. We see
that $G_1^S\diamondsuit(G_2^V\cup G_3^E)$ will be a
\emph{subdivision vertex-edge neighbourhood corona} (see \cite{13})
if $G_3$ is null, and will be a \emph{subdivision edge corona} (see
\cite{10}) if $G_2$ is null. Thus \emph{subdivision vertex-edge
neighbourhood edge-corona} can be viewed as the generalizations of
both \emph{subdivision-edge neighbourhood corona} (denoted by
$G_1\boxminus G_2$) and \emph{subdivision-vertex corona} (denoted by
$G_1\odot G_3$).

\begin{figure}[H]
  \centering
  \includegraphics[width=4.0in]{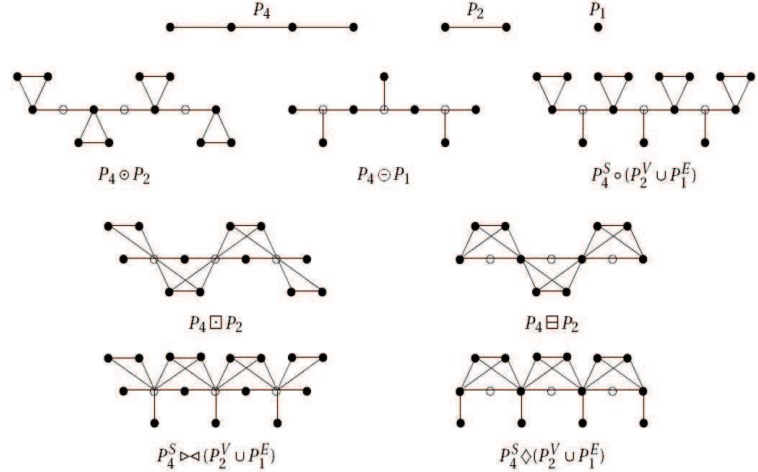}
  \caption {some related
graphs}\label{f-1}
\end{figure}
The \emph{normalized Laplacian matrix} of a graph is introduced in Chung\cite{Chung}. It is a rather new but important tool popularized
by Chung in 1990s. The normalized Laplacian eigenvalues of a graph have a good relationship with other graph invariants for general graphs in a way that other eigenvalues of matrices (such as the adjacency eigenvalues) fail to do. Thus, for a given graph, calculating its normalized Laplacian spectrum as well as formulating the normalized Laplacian characteristic polynomial is a fundamental and very meaningful work in spectral graph theory. Recently, several graph operations such as \emph{subdivisions}\cite{xie}, \emph{$R$-graph}\cite{Yu}, their normalized Laplacian spectra were computed. In 2017, Das and Panigrahi in \cite{Das} have determined the normalized Laplacian spectra of \emph{subdivision-vertex(edge)corona}\cite{10} and \emph{subdivision-vertex (edge) neighbourhood corona}\cite{13}.

In this paper, we focus on determining the normalized Laplacian spectra of $G_1^S\bowtie (G_2^V\cup G_3^E)$ and $G_1^S\diamondsuit(G_2^V\cup G_3^E)$ in terms of the corresponding normalized Laplacian spectra of
connected regular graphs $G_{1}$, $G_{2}$ and $G_{3}$, which extends the corresponding results of \cite{Das}. As applications, these results enable us to construct infinitely many pairs of $\mathcal{L}$-cospectral graphs. Moreover, we also give the number of the
spanning trees, the multiplicative degree-Kirchhoff index
and Kemeny's constant of $G_1^S\bowtie (G_2^V\cup G_3^E)$ (
$G_1^S\diamondsuit(G_2^V\cup G_3^E)$, respectively).
%%%%%%%%%%%%%%%%%%%%%%%%%%%%%%%%%%%%%%%%%%%%%%%%%%%%%%%%%%%%%%%%%%%%%%%%%%%%%%%%%%%%%%%%%5
\section{Preliminaries }\label{s-2}
In this section, we first list some known results for latter use.
\begin{lem}[\!\cite{1}]\label{2-lem-1}
For a graph $G$, let $R(G)$ and $\ell(G)$ be the incidence
matrix of $G$ and the line graph of $G$, respectively. Then
\begin{equation}\label{eq-1}
R(G)^TR(G)=2I_m+A(\ell(G))
\end{equation}
where $m$ is the number of edges of $G$.
\end{lem}
Note that
\begin{equation}\label{eq-7}
R(G)R(G)^T=D(G)+A(G)=Q(G).
\end{equation}
Since non-zero eigenvalues of $R(G)R(G)^T$ and $R(G)^TR(G)$ are the
same, from the relations (\ref{eq-1}) and (\ref{eq-7}) one can
obtain
\begin{equation}\label{eq-6}
\Phi_{A(\ell(G))}(x)=(x+2)^{m-n}\Phi_{Q(G)}(x+2).
\end{equation}

In particular, if $G$ is $r$--regular graph, then by Lemma
\ref{2-lem-1}, we immediately have the following corollary.
\begin{cor}\label{2-cor-1}
Let $G$ be an $r$--regular graph of order $n$. Then
$$
\begin{array}{ll}
\Phi_{A(\ell(G))}(x)&=(x+2)^{m-n}\cdot\prod\limits_{i=1}^{n}(x-(r-2)-\nu_i(G)),\\
\Phi_{A(\ell(G))}(x)&=(x+2)^{m-n}\cdot\prod\limits_{i=1}^{n}(x-(2r-2)+r\lambda_i(G))\\
\end{array}$$
where $\nu_i(G)$ and $\lambda_i(G)$ arethe eigenvalues of $A(G)$ and $\mathcal{L}(G)$, respectively .
\end{cor}
For a graph matrix $M$ of order $n$, we denote by $\textbf{1}_{n}$
and $\textbf{0}_{n}$ the column vector of size $n$ with all the
entries equal one and all the entries equal $0$, respectively.
$M$-\emph{coronal} $\Gamma_M(\lambda)$ is defined, in \cite{4,11}, to be
the sum of the entries of the matrix $(\lambda I-M)^{-1}$, i.e.,
\begin{equation}\label{eq-2}
\begin{array}{ccc}
\Gamma_M(\lambda)=\textbf{1}^{T}_{n}(\lambda I-M)^{-1}\textbf{1}_{n}.
\end{array}
\end{equation}
If $M$ has constant  row sum $t$, it is easy to verify that
\begin{equation}\label{eq-3}
\begin{array}{ccc}
\Gamma_M(\lambda)=\displaystyle \frac{\textbf{1}^{T}_{n}\textbf{1}_{n}}{\lambda-t}=\frac{n}{\lambda-t}.
\end{array}
\end{equation}
It is well known for invertible matrix $M_1$ and $M_4$ that
\begin{equation}\label{eq-5}
\begin{array}{rl}
\det\left(
\begin{array}{rl}
M_1 & M_2\\
M_3 & M_4\\
\end{array}
\right)&=\det(M_4)\cdot \det(M_1-M_2M_4^{-1}M_3)\\
&=\det(M_1)\cdot \det(M_4-M_3M_1^{-1}M_2).
\end{array}
\end{equation}
where $M_1-M_2M_4^{-1}M_3$ and $M_4-M_3M_1^{-1}M_2$ are called the
\emph{Schur complements} \cite{zhang} of $M_4$ and $M_1$, respectively.

For two matrices $A=(a_{ij})$ and $B=(b_{ij})$, of same size $m\times n$, the \emph{Hadamard product} $A\bullet B=(c_{ij})$ of $A$ and $B$ is
a matrix of the same size $m\times n$ with entries given by $c_{ij}=a_{ij}\times b_{ij}$(entrywise multiplication).
The \emph{Kronecker product} $A\otimes B$ of two matrices
$A=(a_{ij})_{m\times n}$ and $B=(b_{ij})_{p\times q}$ is the
$mp\times nq$ matrix obtained from $A$ by replacing each element
$a_{ij}$ by $a_{ij}B$. This is an associative operation with the
property that $(A\otimes B)^{T}=A^T\otimes B^T$ and $(A\otimes
B)(C\otimes D)=AC\otimes BD$ whenever the products $AC$ and $BD$
exist. The latter implies $(A\otimes B)^{-1}=A^{-1}\otimes B^{-1}$
for nonsingular matrices $A$ and $B$. Moreover, if $A$ and $B$ are
$n\times n$ and $p\times p$ matrices, then $\det(A\otimes
B)=\det(A)^{p}\det(B)^{n}$. The reader is referred to \cite{12} for
other properties of the Kronecker product not mentioned here.

\section{The $\mathcal{L}$-spectra of SVEV-corona and SVEE-corona}
In this section, we mainly determe the normalized Laplacian spectrum of \emph{SVEV-corona} and \emph{SVEE-corona}, respectively. For the sake of convenience, we write $\mathcal{G}=G_1^S\bowtie(G_2^V\cup G_3^E)$ and $\mathfrak{G}=G_1^S\diamondsuit(G_2^V\cup G_3^E)$ for short. And we respectively denote the eigenvalues of $\mathcal{L}(G_1)$, $\mathcal{L}(G_2)$ and $\mathcal{L}(G_3)$ by $\theta_i$, $\mu_j$ and $\eta_k$ for $i,j,k=1,2,\ldots,n$. Those symbols will be persisted in what follows.

\begin{thm}\label{thm-3-1}
Let $G_i$ be an $r_i$-regular graph with $n_i$ vertices and $m_i$ edges, where $i=1,2,3$. Then
\begin{equation}\label{eq-3-1}
\mathcal{L}(\mathcal{G})=\left(
\begin{matrix}
I_{n_1} \!\!&\!\! -aR(G_1) \!\!&\!\! O_{n_1\times n_1n_2} \!\!&\!\! O_{n_1\times m_1n_3} \\
-aR(G_1)^T & I_{m_1} & -R(G_1)^T\otimes b_{n_{2}}^T & -I_{m_1}\otimes c_{n_{3}}^T\\
O_{n_1n_2\times n_1} \!\!&\!\! -R(G_1)\otimes b_{n_{2}} \!&\! I_{n_1}\otimes(\mathcal{L}(G_2)\!\bullet\! B(G_2)) \!\!\!&\!\!\! O_{n_{1}n_2\times m_1n_3}\\
O_{m_{1}n_3\times n_1} \!\!&\!\! -I_{m_1}\otimes c_{n_{3}} \!\!&\!\! O_{m_{1}n_3\times n_1n_2} \!\!&\!\! I_{m_1}\otimes(\mathcal{L}(G_3)\!\bullet\! B(G_3))
\end{matrix}
\right)
\end{equation}
where $J_{n}$ is a all-1 matrix of order $n$, and $a=\frac{1}{\sqrt{r_1(2n_2+n_3+2)}}$ is a constant. Moreover,  $b_{n_2}=\frac{1}{\sqrt{(r_1+r_2)(2n_2+n_3+2)}}\mathbf{1}_{n_2}$, $c_{n_3}=\frac{1}{\sqrt{(r_3+1)(2n_2+n_3+2)}}\mathbf{1}_{n_3}$, $B(G_2)=\frac{r_2}{r_1+r_2}J_{n_{2}}-\frac{r_2-1}{r_1+r_2}I_{n_{2}}$, $B(G_3)=\frac{r_3}{r_3+1}J_{n_{3}}-\frac{r_3-1}{r_3+1}I_{n_{3}}$.

\begin{equation}\label{eq-3-2}
\mathcal{L}(\mathfrak{G})=\left(
\begin{matrix}
I_{n_1} \!&\! -aR(G_1) \!&\! -R(G_1)\otimes b_{n_{2}}^T \!&\! -I_{n_1}\otimes c_{n_{3}}^T \\
-aR(G_1)^T & I_{m_1} & O_{m_1\times m_1n_2} & O_{m_1\times n_1n_3}\\
-R(G_1)^T\otimes b_{n_{2}} \!\!&\!\! O_{m_1n_2\times m_1} \!&\! I_{m_1}\otimes(\mathcal{L}(G_2)\!\bullet\!B(G_2)) \!\!\!&\!\!\! O_{m_{1}n_2\times n_1n_3}\\
-I_{n_1}\otimes c_{n_{3}} \!\!&\!\! O_{n_1n_3\times m_1} \!\!&\!\! O_{n_{1}n_3\times m_1n_2} \!\!&\!\! I_{n_1}\otimes(\mathcal{L}(G_3)\!\bullet\! B(G_3))
\end{matrix}
\right)
\end{equation}
where $b_{n_2}=\frac{1}{\sqrt{(r_2+2)(r_1n_2+r_1+n_3)}}\mathbf{1}_{n_{2}}$, $c_{n_3}=\frac{1}{\sqrt{(r_3+1)(r_1n_2+r_1+n_3)}}\mathbf{1}_{n_{3}}$, $B(G_2)=\frac{r_2}{r_2+2}J_{n_{2}}-\frac{r_2-1}{r_2+2}I_{n_{2}}$, $B(G_3)=\frac{r_3}{r_3+1}J_{n_{3}}-\frac{r_3-1}{r_3+1}I_{n_{3}}$, and $a=\frac{1}{\sqrt{2(r_1n_2+r_1+n_3)}}$ is a constant.
\end{thm}

\begin{proof}
The proof of Eq.(\ref{eq-3-2}) is similar to the Eq.(\ref{eq-3-1}), so it's only necessary to show the result given in Eq.(\ref{eq-3-1}) as follows.

We first label the vertices of $\mathcal{G}$:
$V(G_1)=\{v_1,v_2,\ldots,v_{n_1}\}$,
$I(G_1)=\{e_1,e_2,\ldots,e_{m_1}\}$,
$V(G_2)=\{u_1,u_2,\ldots,u_{n_2}\}$ and
$V(G_3)=\{w_1,w_2,\ldots,w_{n_3}\}$. For $i=1,2,\ldots,n_1$, let
$U_i=\{u_1^i,u_2^i,$ $\ldots,u_{n_2}^i\}$ denote the vertices of the
$i$-th copy of $G_2$ in $\mathcal{G}$, and $W_j=\{w_1^j,$
$w_2^j,\ldots,w_{n_3}^j\}~(j=1,2,\ldots,m_1)$ the $j$-th copy of
$G_3$ in $\mathcal{G}$. Then the vertices of $\mathcal{G}$ are partitioned by
\begin{equation}\label{eq-3-3}
\begin{array}{ccc}
V(G_1)\cup I(G_1)\cup (U_1\cup U_2\cup\ldots\cup
U_{n_1})\cup(W_1\cup W_2\cup\ldots\cup W_{m_1}).
\end{array}
\end{equation}
 By Definition \ref{d-1} we see that
\begin{equation*}\left\{
\begin{array}{rl}
d_{\mathcal{G}}(v_i)&=d_{G_1}(v_i)=r_1, \hspace{5.0cm} i=1,2,\ldots,n_1; \\
d_{\mathcal{G}}(e_i)&=2n_{2}+n_3+2, \hspace{5cm} i=1,2,\ldots,m_1;\\
d_{\mathcal{G}}(u_j^{i})&=d_{G_2}(u_j)+d_{G_1}(v_i)=r_{2}+r_1,  \hspace{2cm}
j=1,2,\ldots,n_2,i=1,2,\ldots,n_1;\\
d_{\mathcal{G}}(w_j^{i})&=d_{G_3}(w_j)+1=r_3+1, \hspace{3cm} j=1,2,\ldots,n_3,
i=1,2,\ldots,m_1.
\end{array}\right.
\end{equation*}
Let $R(G_{1})$ be the vertex-edge incidence matrix of $G_1$. Then the
adjacency matrix of $\mathcal{G}$ can be
represented in the form of block-matrix according to the ordering
of (\ref{eq-3-3}) in the
following.
\begin{equation*}
A(\mathcal{G})=\left(
\begin{matrix}
O_{n_{1}\times n_{1}} ~& R(G_1) ~& O_{n_1\times n_1n_2} ~& O_{n_1\times m_1n_3} \\
R(G_1)^T ~& O_{m_{1}\times m_{1}} ~& R(G_1)^T\otimes\textbf{1}_{n_{2}}^T ~& I_{m_1}\otimes\textbf{1}_{n_{3}}^T\\
O_{n_1n_2\times n_1} ~& R(G_1)\otimes\textbf{1}_{n_{2}} ~& I_{n_1}\otimes A(G_2) ~& O_{n_{1}n_2\times m_1n_3}\\
O_{m_{1}n_3\times n_1} ~&
I_{m_1}\otimes\textbf{1}_{n_{3}}
~& O_{m_{1}n_3\times n_1n_2} ~& I_{m_1}\otimes A(G_3)
\end{matrix}
\right),
\end{equation*} and
\begin{equation*}
D(\mathcal{G})=\left(
\begin{matrix}
r_1I_{n_1} \!\!\!&\!\!\!   \!\!\!&\!\!\!  \!\!\!&\!\!\! \\
 \!\!\!&\!\!\!  (2n_2+n_3+2)I_{m_1} \!\!\!&\!\!\!  \!\!\!&\!\!\! \\
 \!\!&\!\!   \!\!&\!\! (r_1+r_2)I_{n_1n_2} \!\!&\!\! \\
 \!\!&\!\!   \!\!&\!\!  \!\!&\!\! (r_3+1)I_{m_1n_3}
\end{matrix}
\right)
\end{equation*}
where $A(G_{2})$ and $A(G_{3})$ represent the adjacency matrices of
$G_{2}$ and $G_{3}$, respectively.

Since $G_2$ is an $r_2$-regular graph, we have $\mathcal{L}(G_2)=I_{n_2}-\frac{1}{r_2}A(G_2)$. Thus,
$$\mathcal{L}(G_2)\bullet B(G_2)=(I_{n_2}-\frac{1}{r_2}A(G_2))\bullet B(G_2)=I_{n_2}-\frac{1}{r_1+r_2}A(G_2).$$
By direct computation, one can obtain that
$$I_{n_1n_2}-\frac{1}{r_1+r_2}I_{n_1}\otimes A(G_2)=I_{n_1}\otimes(\mathcal{L}(G_2)\bullet B(G_2)).$$
Furthermore, we can obtain that
$$I_{m_1n_3}-\frac{1}{r_3+1}I_{m_1}\otimes A(G_3)=I_{m_1}\otimes(\mathcal{L}(G_3)\bullet B(G_3)).$$
By $\mathcal{L}(\mathcal{G})=I-D(\mathcal{G})^{-1/2}A(\mathcal{G})D(\mathcal{G})^{-1/2}$, the required normalized Laplacian matrix is given below:
\begin{equation*}
\mathcal{L}(\mathcal{G})=\left(
\begin{matrix}
I_{n_1} \!&\! -aR(G_1) \!&\! O_{n_1\times n_1n_2} \!&\! O_{n_1\times m_1n_3} \\
-aR(G_1)^T & I_{m_1} & -R(G_1)^T\otimes b_{n_{2}}^T & -I_{m_1}\otimes c_{n_{3}}^T\\
O_{n_1n_2\times n_1} \!\!&\!\! -R(G_1)\otimes b_{n_{2}} \!&\! I_{n_1}\otimes(\mathcal{L}(G_2)\!\bullet\! B(G_2)) \!\!\!&\!\!\! O_{n_{1}n_2\times m_1n_3}\\
O_{m_{1}n_3\times n_1} \!\!&\!\! -I_{m_1}\otimes c_{n_{3}} \!\!&\!\! O_{m_{1}n_3\times n_1n_2} \!\!&\!\! I_{m_1}\otimes(\mathcal{L}(G_3)\!\bullet\! B(G_3))
\end{matrix}
\right).
\end{equation*}
The proof is completed.
\end{proof}
\begin{remark}\label{remark-3-1}
Obviously, using the same partition of \emph{(\ref{eq-3-3})} to the graph $\mathfrak{G}$, one can obtain
\begin{equation*}
\left\{\begin{array}{rl}
d_{\mathfrak{G}}(e_i)& =2, \ \hspace{6.5cm}  i=1,2,\ldots,m_1;\\
d_{\mathfrak{G}}(u_j^{i})& =d_{G_2}(u_j)+2=r_{2}+2, \hspace{3.5cm}
j=1,2,\ldots,n_2,i=1,2,\ldots,n_1;\\
d_{\mathfrak{G}}(w_j^{i})& =d_{G_3}(w_j)+1=r_{3}+1, \hspace{3.5cm} j=1,2,\ldots,n_3, i=1,2,\ldots,m_1;\\
d_{\mathfrak{G}}(v_i)& =(n_{2}+1)d_{G_1}(v_i)+n_{3}=(n_{2}+1)r_{1}+n_{3}, \hspace{0.6cm} i=1,2,\ldots,n_1;\\
\end{array}\right.
\end{equation*} and
\begin{equation*}
A(\mathfrak{G})=\left(
\begin{matrix}
O_{n_{1}\times n_{1}} ~& R(G_1) ~& R(G_1)\otimes~\textbf{1}_{n_{2}}^{T} ~& I_{n_{1}}~\otimes~\textbf{1}^{T}_{n_{3}} \\
R^T(G_1) ~& O_{m_{1}\times m_{1}} ~& I_{m_{1}}\otimes~\textbf{0}_{n_{2}}^{T} ~& R^{T}(G_1)\otimes~\textbf{0}^{T}_{n_{3}}\\
R^{T}(G_1)\otimes~\textbf{1}_{n_{2}}~& I_{m_{1}}~\otimes~\textbf{0}_{n_{2}} ~& I_{m_1}\otimes~A(G_2) ~& R^{T}(G_1)~\otimes~O_{n_{2}\times n_{3}}^{T}\\
I_{n_{1}}\otimes~\textbf{1}_{n_{3}} ~& R(G_1)\otimes~\textbf{0}_{n_{3}}
~&R(G_1)\otimes~O_{n_{3}\times n_{2}}~& I_{n_1}\otimes~A(G_3)
\end{matrix}
\right).
\end{equation*}
By immediate calculation, the $\mathcal{L}(\mathfrak{G})$ follows.
\end{remark}
\begin{thm}\label{thm-3-2}
Let $G_i$ be an $r_i$-regular graph with $n_i$ vertices and $m_i$ edges, where $i=1,2,3$. Then
\begin{equation*}
\begin{array}{rl}
\emph{(i)}\hspace{0.2cm} \Phi_{\mathcal{L}(\mathcal{G})}(\lambda)\!\!\!\!&=\!\displaystyle(\lambda\!\!-\!\!1\!\!
-\!\frac{n_3}{(r_3\!\!+\!\!1)(2n_2\!\!+\!\!n_3\!\!+\!\!2)(\lambda\!\!
-\!\!\frac{1}{r_3+1})})^{m_1\!-\!n_1}\!\!\cdot\!\!\textstyle\prod\limits_{j=1}^{n_2}
(\lambda\!\!-\!\frac{r_1+r_2\mu_j}{r_1+r_2})^{n_1}\!\!\cdot\!\!
\prod\limits_{k=1}^{n_3}(\lambda\!\!-\!\frac{1+r_3\eta_k}{r_3+1})^{m_1}\\
&\times\!\textstyle\prod\limits_{i=1}^{n_1}\big((\lambda\!-\!1)
(\lambda\!-\!1\!-\!\frac{n_3}{(r_3+1)(2n_2+n_3+2)(\lambda-\frac{1}{r_3+1})})\!
-\!\frac{(2-\theta_i)((r_1+r_2+n_2r_1)\lambda-(1+n_2)r_1)}{(r_1+r_2)(2n_2+n_3+2)
(\lambda-\frac{r_1}{r_1+r_2})}\big);
\end{array}
\end{equation*}
\begin{equation*}
\begin{array}{rl}
\hspace{-1.2cm}\emph{(ii)}\hspace{0.2cm}\Phi_{\mathcal{L}(\mathfrak{G})}(\lambda)\!\!\!\!&=
\!\displaystyle(\lambda\!-\!1)^{m_1\!-\!n_1}\!\cdot\!\textstyle
\prod\limits_{j=1}^{n_2}(\lambda\!\!-\!\frac{2+r_2\mu_j}{r_2+2})^{m_1}\!\cdot
\!\prod\limits_{k=1}^{n_3}(\lambda\!\!-\!\frac{1+r_3\eta_k}{r_3+1})^{n_1}\\
&\times\!\textstyle\prod\limits_{i=1}^{n_1}\big((\lambda\!-\!1)
(\lambda\!-\!1\!-\!\frac{n_3}{(r_3+1)(r_1n_2+n_3+r_1)(\lambda-\frac{1}{r_3+1})})
\!-\!\frac{r_1(2-\theta_i)((2n_2+r_2+2)\lambda-2n_2-2)}{2(r_2+2)(r_1n_2+n_3+r_1)
(\lambda-\frac{2}{r_2+2})}\big)
\end{array}
\end{equation*}
where $\theta_i$, $\mu_j$ and $\eta_k$ are the eigenvalues of $\mathcal{L}(G_1)$, $\mathcal{L}(G_2)$ and $\mathcal{L}(G_3)$, respectively.
\end{thm}

\begin{proof}
According to Theorem \ref{thm-3-1}, the normalized Laplacian characteristic polynomial of $\mathcal{G}$ is
$$\Phi_{\mathcal{L}(\mathcal{G})}(\lambda)=\det(\lambda I_n-\mathcal{L}(\mathcal{G}))=\det(B_0),$$
where
\begin{equation*}
B_{0}=\left(
\begin{smallmatrix}
(\lambda-1)I_{n_{1}} ~~~~&~~~~ aR(G_1) ~~~~&~~~~ O ~~~~&~~~~ O \\
aR(G_1)^T & (\lambda-1)I_{m_1} & R(G_1)^T\otimes b_{n_{2}}^T & I_{m_1}\otimes c_{n_{3}}^T\\
O & R(G_1)\otimes b_{n_{2}} & I_{n_1}\otimes(\lambda I_{n_2}-\mathcal{L}(G_2)\bullet B(G_2)) & O \\
O & I_{m_1}\otimes c_{n_{3}} & O & I_{m_1}\otimes(\lambda I_{n_3}-\mathcal{L}(G_3)\bullet B(G_3))
\end{smallmatrix}
\right).
\end{equation*}

We write $P$ as the elementary block matrix below
\begin{equation*}
P=\left(
\begin{smallmatrix}
I_{n_{1}} \!&\! O \!&\! O \!&\! O \\
O & I_{m_{1}} ~& -R(G_1)^T\otimes( b_{n_{2}}^{T}(\lambda I_{n_2}-\mathcal{L}(G_2)\bullet B(G_2))^{-1}) ~& -I_{m_1}\otimes( c_{n_{3}}^{T}(\lambda I_{n_3}-\mathcal{L}(G_3)\bullet B(G_3))^{-1})\\
O & O & I_{n_1}\otimes I_{n_2} & O \\
O & O & O & I_{m_1}\otimes I_{n_3}
\end{smallmatrix}
\right).
\end{equation*}
Let $B=PB_{0}$. Then
\begin{equation*}
B=\left(
\begin{smallmatrix}
(\lambda-1)I_{n_{1}} \!\!&\!\! aR(G_1) \!\!&\!\! O \!\!&\!\! O \\
aR(G_1)^T & (\lambda-1-\Gamma_{3}(\lambda))I_{m_{1}}\!-\Gamma_{2}(\lambda)R(G_1)^TR(G_1) & O & O \\
O \!&\! R(G_1)\otimes b_{n_{2}} \!\!&\!\! I_{n_1}\otimes(\lambda I_{n_2}\!-\mathcal{L}(G_2)\bullet B(G_2)) \!&\! O \\
O \!&\! I_{m_1}\otimes c_{n_{3}} \!&\! O \!\!&\!\! I_{m_1}\otimes(\lambda I_{n_3}\!-\mathcal{L}(G_3)\bullet B(G_3))
\end{smallmatrix}
\right)
\end{equation*}
where $\Gamma_2(\lambda)=b_{n_2}^T(\lambda I_{n_2}-\mathcal{L}(G_2)\bullet B(G_2))^{-1}b_{n_2}$ and $\Gamma_3(\lambda)=c_{n_3}^T(\lambda I_{n_3}-\mathcal{L}(G_3)\bullet B(G_3))^{-1}c_{n_3}$.

Note that $\det(P)=1$. Then we have
$$\Phi_{\mathcal{L}(\mathcal{G})}(\lambda)=\det(B_0)=\det(P^{-1})\det(B)=\det(B).$$
For the matrix $B$, one can get
$$\det(B)\!\!=\!\!\det\big(I_{n_1}\!\!\otimes(\lambda I_{n_2}\!\!-\!\mathcal{L}(G_2)\bullet B(G_2))\big)\cdot \det\big(I_{m_1}\!\!\otimes(\lambda I_{n_3}\!\!-\!\mathcal{L}(G_3)\bullet B(G_3))\big)\cdot \det(S_1),$$
where
\begin{equation*}
S_1=\left(
\begin{matrix}
(\lambda-1)I_{n_{1}} & aR(G_1) \\
 aR(G_1)^T  &  (\lambda-1-\Gamma_3(\lambda))I_{m_{1}}\!-\Gamma_2(\lambda)R(G_1)^TR(G_1)
\end{matrix}
\right).
\end{equation*}
Let $\theta_i$, $\mu_j$ and $\eta_k$ be the eigenvalues of $\mathcal{L}(G_1)$, $\mathcal{L}(G_2)$ and $\mathcal{L}(G_3)$, respectively.
By applying Eq.(\ref{eq-5}), the following result follows from Corollary \ref{2-cor-1} that
\begin{equation*}
\begin{array}{rl}
\det (S_1)&=\left|
\begin{array}{cc}
(\lambda-1)I_{n_{1}} & aR(G_1) \\\vspace{0.2cm}
 aR(G_1)^T  &  (\lambda-1-\Gamma_{3}(\lambda))I_{m_{1}}\!-\Gamma_{2}(\lambda)R(G_1)^TR(G_1)
\end{array}\right|\\
&=\displaystyle \det((\lambda-1)I_{n_{1}})\cdot \det((\lambda-1-\Gamma_{3}(\lambda))I_{m_{1}}-(\Gamma_{2}(\lambda)+\frac{a^2}{\lambda-1})R(G_1)^TR(G_1)).\\
&=\displaystyle (\lambda-1)^{n_1}\det\big((\lambda-1-\Gamma_{3}(\lambda))I_{m_{1}}-(\Gamma_{2}(\lambda)+\frac{a^2}{\lambda-1})(A(\ell(G_1))+2I_{m_1})\big)\\
&=\displaystyle (\lambda-1)^{n_1}\cdot\big((\lambda-1-\Gamma_{3}(\lambda))-(\Gamma_{2}(\lambda)+\frac{a^2}{\lambda-1})(-2+2)\big)^{m_1-n_1}\\\vspace{0.2cm}
&\hspace{0.5cm}\times\displaystyle ~\det\big((\lambda-1-\Gamma_{3}(\lambda))I_{n_{1}}-(\Gamma_{2}(\lambda)+\frac{a^2}{\lambda-1})(2r_1I_{n_1}-r_1\mathcal{L}(G_1)\big)\\
&=\displaystyle (\lambda\!\!-\!1\!\!-\!\!\Gamma_{3}(\lambda))^{m_1\!-\!n_1}\!\cdot\! \det\big((\lambda\!-\!1)(\lambda\!\!-\!1\!\!-\!\!\Gamma_{3}(\lambda))I_{n_1}\!\!\!-\!r_1((\lambda\!-\!1)\Gamma_{2}(\lambda)\!+\!a^2)(2I_{n_1}\!\!\!-\!\!\mathcal{L}(G_1))\big)\\
&=\displaystyle(\lambda\!\!-\!1\!\!-\!\!\Gamma_{3}(\lambda))^{m_1\!-\!n_1}\!\cdot\! \textstyle\prod\limits_{i=1}^{n_1}\big((\lambda\!-\!1)(\lambda\!\!-\!1\!\!-\!\!\Gamma_{3}(\lambda))\!-\!r_1((\lambda\!-\!1)\Gamma_{2}(\lambda)\!+\!\frac{1}{r_1(2n_2+n_3+2)})(2\!-\!\theta_i)\big).
\end{array}
\end{equation*}

Since $\mathcal{L}(G_2)\bullet B(G_2)=I_{n_2}-\frac{1}{r_1+r_2}A(G_2)$ and $A(G_2)=r_2(I_{n_2}-\mathcal{L}(G_2))$, we get $$\mathcal{L}(G_2)\bullet B(G_2)=\frac{1}{r_1+r_2}(r_1I_{n_2}+r_2\mathcal{L}(G_2)).$$
Similarly, $\mathcal{L}(G_3)\bullet B(G_3)=\frac{1}{r_3+1}(I_{n_3}+r_3\mathcal{L}(G_3))$.\vspace{0.2cm}

Also since $G_2$ is $r_{2}$-regular, the sum of all entries on every row of its normalized Laplacian matrix is zero. In other words, $\mathcal{L}(G_2)b_{n_2}=(1-\frac{r_2}{r_2})b_{n_2}=0\cdot b_{n_2}=\textbf{0}$. Then $(\mathcal{L}(G_2)\bullet B(G_2))b_{n_2}=(1-\frac{r_2}{r_1+r_2})b_{n_2}=\frac{r_1}{r_1+r_2}b_{n_2}$ and $(\lambda I_{n_2}-(\mathcal{L}(G_2)\bullet B(G_2)))b_{n_2}=(\lambda-\frac{r_1}{r_1+r_2})b_{n_2}$. Also, $b_{n_2}^Tb_{n_2}=\frac{n_2}{(r_1+r_2)(2n_2+n_3+2)}$. Thus
$$\Gamma_{2}(\lambda)=b_{n_2}^T(\lambda I_{n_2}-\mathcal{L}(G_2)\bullet B(G_2))^{-1}b_{n_2}=\frac{b_{n_2}^Tb_{n_2}}{\lambda-\frac{r_1}{r_1+r_2}}=\frac{n_2}{(r_1+r_2)(2n_2+n_3+2)(\lambda-\frac{r_1}{r_1+r_2})}.$$
The value of $\Gamma_{3}(\lambda)$ is similar to that of $\Gamma_{2}(\lambda)$, and so,
$$\Gamma_{3}(\lambda)=c_{n_3}^T(\lambda I_{n_3}-\mathcal{L}(G_3)\bullet B(G_3))^{-1}c_{n_3}=\frac{c_{n_3}^Tc_{n_3}}{\lambda-\frac{1}{r_3+1}}=\frac{n_3}{(r_3+1)(2n_2+n_3+2)(\lambda-\frac{1}{r_3+1})}.$$

In summary, the normalized Laplacian characteristic polynomial of $\mathcal{G}$ is
\begin{equation*}
\begin{array}{rl}
\Phi_{\mathcal{L}(\mathcal{G})}(\lambda)\!\!\!\!&=\displaystyle\textstyle\prod\limits_{j=1}^{n_2}(\lambda-\frac{r_1+r_2\mu_j}{r_1+r_2})^{n_1}\cdot\prod\limits_{k=1}^{n_3}(\lambda-\frac{1+r_3\eta_k}{r_3+1})^{m_1}\cdot\det (S_1)\\
&=\!\displaystyle(\lambda\!\!-\!\!1\!\!-\!\frac{n_3}{(r_3\!\!+\!\!1)
(2n_2\!\!+\!\!n_3\!\!+\!\!2)(\lambda\!\!-\!\!\frac{1}{r_3+1})})^{m_1\!-\!n_1}
\!\!\cdot\!\!\textstyle\prod\limits_{j=1}^{n_2}(\lambda\!\!-\!\frac{r_1+r_2\mu_j}
{r_1+r_2})^{n_1}\!\!\cdot\!\!\prod\limits_{k=1}^{n_3}(\lambda\!\!-
\!\frac{1+r_3\eta_k}{r_3+1})^{m_1}\\
&\times\!\textstyle\prod\limits_{i=1}^{n_1}\big((\lambda\!-\!1)
(\lambda\!-\!1\!-\!\frac{n_3}{(r_3+1)(2n_2+n_3+2)(\lambda-\frac{1}{r_3+1})})\!
-\!\frac{(2-\theta_i)((r_1+r_2+n_2r_1)\lambda-(1+n_2)r_1)}{(r_1+r_2)(2n_2+n_3+2)
(\lambda-\frac{r_1}{r_1+r_2})}\big),
\end{array}
\end{equation*}
as required. The proof of (ii) is similar to the (i).
\end{proof}
\begin{remark}\label{remark-3-2}
In Theorem \ref{thm-3-2}, if one of graphs $G_{2}$ and $G_{3}$ is null in $\mathcal{G}$ or $\mathfrak{G}$, then one can directly deduce the main results(see Theorems 2.2-2.5) due to A. Das and P. Panigrahi in \cite{Das}. For the simplicity, we here omit Theorems 2.2-2.5 belonged to A. Das and P. Panigrahi\cite{Das}.
\end{remark}

From Threom \ref{thm-3-2}, one can easily obtain the following corollaries.
\begin{cor}\label{cor-3-1}
Let $G_i$ be an $r_i$-regular graph with $n_i$ vertices and $m_i$ edges for $i=1,2,3$. Then the normalized Laplacian spectrum of $\mathcal{G}$ consist of:
\begin{enumerate}
\item [\emph{(a)}] $\frac{r_1+r_2\mu_j}{r_1+r_2}$ repeated $n_1$ times for each eigenvalue $\mu_j$ of $\mathcal{L}(G_2)$, $j=1,2, \ldots,n_2-1$;
\item [\emph{(b)}] $\frac{1+r_3\eta_k}{r_3+1}$ repeated $m_1$ times for each eigenvalue $\eta_k$ of $\mathcal{L}(G_3)$, $k=1,2, \ldots,n_3-1$;
\item [\emph{(c)}] two roots of the equation
\begin{equation}\label{eq-3.1-c}
\begin{array}{rl}
&(2n_2r_3+n_3r_3+2r_3+2n_2+n_3+2)\lambda^2-(2n_2r_3+n_3r_3+2r_3+4n_2\\
&+2n_3+4)\lambda+2n_2+2=0
\end{array}
\end{equation}
where each root repeats $m_1-n_1$ times;
\item [\emph{(d)}] four roots of the equation
\begin{equation}
\begin{array}{rl}\label{eq-3.1-d}
&(2n_2+n_3+2)((1+r_3)\lambda-1)((r_1+r_2)\lambda-r_1)(\lambda-1)^2-n_3(\lambda-1)((r_1+r_2)\lambda\\
&-r_1)-(2-\theta_i)\big((r_1+r_2+n_2r_1)\lambda-(r_1+r_1n_2)\big)((1+r_3)\lambda-1)=0,
\end{array}
\end{equation}
 where each eigenvalue $\theta_i$ of $\mathcal{L}(G_1)$, $i=1,2,\ldots,n_1$.
\end{enumerate}
\end{cor}

\begin{cor}\label{cor-3-2}
If $G_i$ is an $r_i$-regular graph with $n_i$ vertices and $m_i$ edges for $i=1,2,3$, the normalized Laplacian spectrum of $\mathfrak{G}$ consist of:
\begin{enumerate}
\item [\emph{(a)}] $\frac{2+r_2\mu_j}{r_2+2}$ repeated $m_1$ times for each eigenvalue $\mu_j$ of $\mathcal{L}(G_2)$, $j=1,2, \ldots,n_2-1$;
\item [\emph{(b)}] $\frac{1+r_3\eta_k}{r_3+1}$ repeated $n_1$ times for each eigenvalue $\eta_k$ of $\mathcal{L}(G_3)$, $k=1,2, \ldots,n_3-1$;
\item [\emph{(c)}] $1$ repeats $m_1-n_1$ times, $\frac{2}{r_2+2}$ repeats $m_1-n_1$ times;
\item [\emph{(d)}] four roots of the equation
\begin{equation*}
\begin{array}{rl}
&2(r_1n_2+n_3+r_1)((1+r_3)\lambda-1)((2+r_2)\lambda-2)(\lambda-1)^2-2n_3(\lambda-1)((2+r_2)\lambda-2)\\
&-r_1(2-\theta_i)\big((2n_2+r_2+2)\lambda-2n_2-2\big)((1+r_3)\lambda-1)=0,
\end{array}
\end{equation*}
 where each eigenvalue $\theta_i$ of $\mathcal{L}(G_1)$, $i=1,2,\ldots,n_1$.
\end{enumerate}
\end{cor}
\begin{example}\label{examp-3-1}
Let $G=C_4^S\bowtie(K_2^V\cup K_{2}^E)$ and $H=C_4^S\diamondsuit(K_{2}^V\cup
K_{1}^E)$ (shown in Fig.\ref{fig-3.2}). By simple computation, $Spec_{\mathcal{L}}(C_{4})=\{0,1,1,2\}$, $Spec_{\mathcal{L}}(K_{2})=\{0,2\}$ and $Spec_{\mathcal{L}}(K_{1})=\{0\}$.

From Corollary \ref{cor-3-1}, the normalized Laplacian spectrum of $G$ consist of: $\frac{4}{3}$ (multiplicity 4), $\frac{3}{2}$ (multiplicity 4), four roots of the equation $24x^4-76x^3+71x^2-20x=0$, four roots (multiplicity 2) of the equation $48x^4-152x^3+156x^2-59x+6=0$, and four roots of the equation $48x^4-152x^3+170x^2-78x+12=0$.

From Corollary \ref{cor-3-2}, the normalized Laplacian spectrum of $H$ consist of: $\frac{4}{3}$ (multiplicity 4), four roots of the equation $42x^4-154x^3+176x^2-64x=0$, four roots (multiplicity 2) of the equation $42x^4-154x^3+190x^2-90x+12=0$, and four roots of the equation $42x^4-154x^3+204x^2-116x+24=0$.

\begin{figure}[H]
  \centering
  \subfigure{
    \includegraphics[width=3.8in]{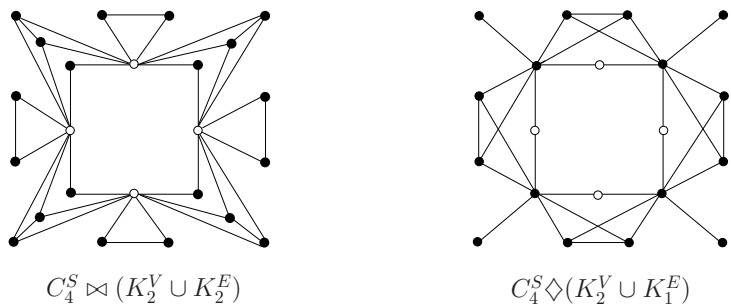}}
  \caption{$G=C_4^S\bowtie(K_2^V\cup K_{2}^E)$ and $H=C_4^S\diamondsuit(K_{2}^V\cup
K_{1}^E)$}\label{fig-3.2}
\end{figure}

\end{example}

%%%%%%%%%%%%%%%%%%%%%%%%%%%%%%%%%%%%%%%%%%%%%%%%%%%%%%%%%%%%%%%%%%%%%%%%%%%%%%%%%%%%%%%
\section{Applications}
In this section, we will give four distinct applications, such as construction for $\mathcal{L}$-cospectral graphs, computation for the number of \emph{spanning trees}, \emph{the multiplicative degree-Kirchhoff index} and \emph{Kemeny's constant} on SVEV-corona and SVEE-corona respectively.

\subsection {Construct $\mathcal{L}$-cospectral graphs}
In \cite{ER}, Dam and Haemers have proposed \emph{`Which graphs
are determined by their spectra?'}. The
question for the \emph{normalized Laplacian} spectrum is also one of the outstanding unsolved problems in the theory of graph spectra. Thus, if one wish to settle the question for graphs in general, it is natural to look
for constructing pairs of $\mathcal{L}$-cospectral graphs. In this section, we will respectively construct many infinite families of pairs of $\mathcal{L}$-cospectral graphs from \emph{SVEV-corona} and \emph{SVEE-corona}, which are generalized Theorem 2.7 due to Das and Panigrahi in \cite{Das}.

\begin{thm}\label{thm-4.1-1}
Let $G_{i}$ and $H_{i}$ (not necessarily distinct isomorphic) are pairwise $\mathcal{L}$-cospectral \textcolor[rgb]{0.00,0.00,1.00}{regular} graphs for $i=1,2,3$. Then
\begin{enumerate}
\item [\emph{(1)}] $G_1^S\bowtie(G_2^V\cup G_3^E)$ and $H_1^S\bowtie(H_2^V\cup H_3^E)$ are $\mathcal{L}$-cospectral graphs;
\item [\emph{(2)}] $G_1^S\diamondsuit(G_2^V\cup G_3^E)$ and $H_1^S\diamondsuit(H_2^V\cup H_3^E)$ are $\mathcal{L}$-cospectral graphs.
\end{enumerate}
\end{thm}
\begin{proof}
From Theorem \ref{thm-3-2} we know that, the normalized Laplacian spectra of $G_1^S\bowtie(G_2^V\cup G_3^E)$ and $G_1^S\diamondsuit$ $(G_2^V\cup G_3^E)$ are completely determined by the degrees of regularities, the number of vertices, the number of edges and the normalized Laplacian spectra of regular graphs $G_i$ $(i=1,2,3)$. So the conclusions follows.
\end{proof}
\begin{remark}\label{remark-4.1-1}
Let $G_{1}$ and $H_{1}$ be two $\mathcal{L}$-cospectral graphs. If $\{G_{i}^{(1)},G_{i}^{(2)},G_{i}^{(3)},\ldots\}$ and $\{H_{i}^{(1)},H_{i}^{(2)},$ $H_{i}^{(3)},\ldots\}$ are two pairwise $\mathcal{L}$-cospectral
but not isomorphic graph sequences for $i=2,3$. We wonder what conditions such that the largest infinite families of pairs of $(G_1^S\bowtie(G_2^{(1)V}\cup G_3^{(1)E}))\bowtie(G_2^{(2)V}\cup G_3^{(2)E})\cdots$ and $(H_1^S\bowtie(H_2^{(1)V}\cup H_3^{(1)E}))\bowtie(H_2^{(2)V}\cup H_3^{(2)E})\cdots$,  and $(G_1^S\diamondsuit(G_2^{(1)V}\cup G_3^{(1)E}))\diamondsuit(G_2^{(2)V}\cup G_3^{(2)E})\cdots$ and $(H_1^S\diamondsuit(H_2^{(1)V}\cup H_3^{(1)E}))$ $\diamondsuit(H_2^{(2)V}\cup H_3^{(2)E})\cdots$ are $\mathcal{L}$-cospectral ?
\end{remark}
\begin{example}\label{examp-4.1-1}
Let $G_1$ and $H_1$ be two graphs shown in Fig.\ref{fig-2}. Then by Matlab 7.0 one can get $\Phi_{A(G_{1})}(x)=\Phi_{A(H_{1})}(x)=x^{14}- 21x^{12}-2x^{11}+164x^{10} +22x^{9}-599x^{8} -88x^{7}+1047x^{6}+168x^{5}
-800x^{4}-160x^{3}+216x^{2}+40x-12$.
It is easy to see that $G_1$ and $H_1$ are $A$-cospectral but not isomorphic with each other. Note that $\mathcal{L}=I-D^{-1/2}AD^{-1/2}$. And so, two graphs are $A$-cospectral imply that they are $\mathcal{L}$-cospectral. Consequently, it follows from Theorem \ref{thm-4.1-1} that $G_1^S\bowtie(K_3^V\cup K_2^E)$ and $H_1^S\bowtie(K_3^V\cup K_2^E)$ are $\mathcal{L}$-cospectral graphs, so are $G_1^S\diamondsuit(K_3^V\cup K_2^E)$ and $H_1^S\diamondsuit(K_3^V\cup K_2^E)$, see Fig.\ref{fig-3} and Fig.\ref{fig-4} for instance.
\end{example}
\begin{figure}[H]
  \centering
  \includegraphics[width=3.2in]{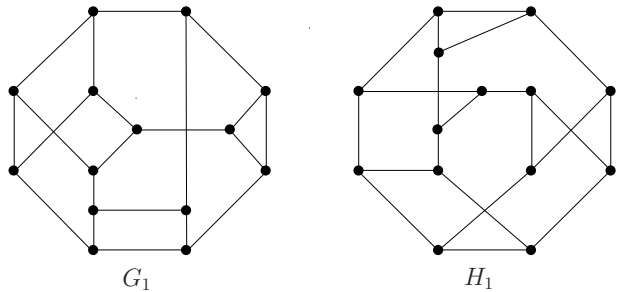}
  \caption{$G_1$ and $H_1$} \label{fig-2}
\end{figure}
\begin{figure}[H]
  \centering
  \subfigure{
    \includegraphics[width=2.0in]{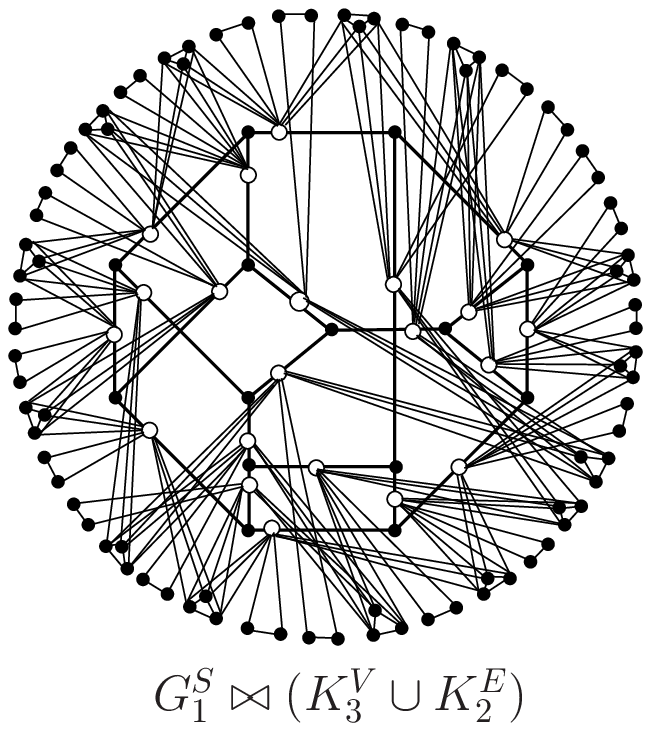}}
  \hspace{1in}
  \subfigure{
    \includegraphics[width=2.0in]{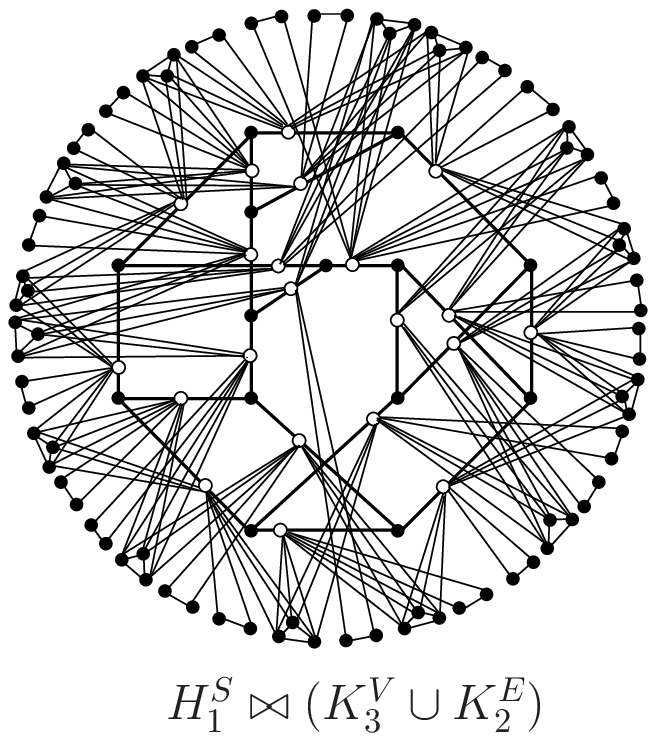}}
  \caption{$G_1^S\bowtie(K_3^V\cup K_2^E)$ and $H_1^S\bowtie(K_3^V\cup K_2^E)$}\label{fig-3}
\end{figure}
\begin{figure}[H]
  \centering
  \subfigure{
    \includegraphics[width=2.0in]{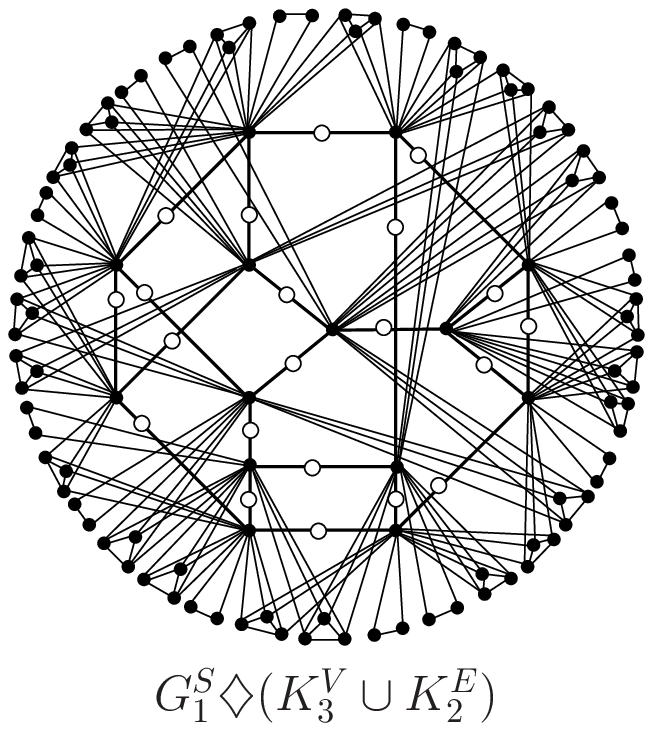}}
  \hspace{1in}
  \subfigure{
    \includegraphics[width=2.0in]{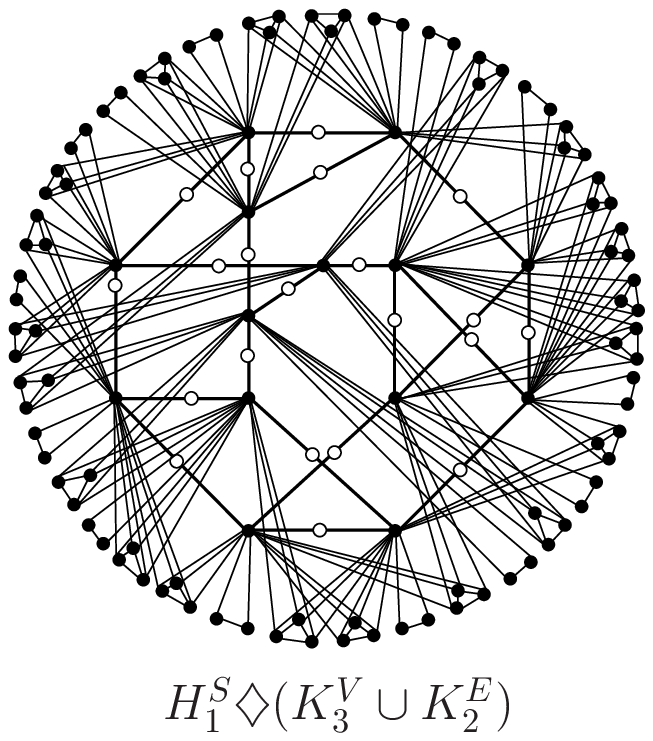}}
  \caption{$G_1^S\diamondsuit(K_3^V\cup K_2^E)$ and $H_1^S\diamondsuit(K_3^V\cup K_2^E)$}\label{fig-4}
\end{figure}

%%%%%%%%%%%%%%%%%%%%%%%%%%%%%%%%%%%%%%%%%%%%%%%%%%%%%%%%%%%%%%%%%%%%%%%%%%%%%%%%%%%%%%%%%%%%%%%%%%%
\subsection{The number of spanning trees}
Let $G$ be a connected graph of order $n$. A \emph{spanning tree} is a spanning subgraph of $G$ that is a tree. A known result from Chung \cite{Chung} allows the calculation of this number from the normalized Laplacian spectrum and the degrees of all the vertices, thus the number of spanning trees $\tau(G)$ of connected graph $G$ is
\begin{equation}\label{eq-4.1-1}
\hspace{4.8cm}\tau(G)=\displaystyle\frac{\Pi_{i=1}^{n}d_{i}
\Pi_{i=1}^{n-1}\lambda_{i}}{\sum_{i=1}^{n}d_{i}}.
\end{equation}
\begin{thm}\label{thm-4.2-1}
Let $G_i$ be an $r_i$-regular graph with $n_i$ vertices and $m_i$ edges for $i=1,2,3$. Then
\begin{equation*}
 \begin{array}{rl}
 \hspace{-1cm}(1)~~\tau(\mathcal{G})&= \prod\limits_{i=1}^{n_1-1}\!\theta_i\cdot\!\prod\limits_{j=1}^{n_2-1}\!(r_1\!+\!r_2\mu_j)^{n_1}\!\cdot\!\prod\limits_{k=1}^{n_3-1}\!(1\!+\!r_3\eta_k)^{m_1}\!\cdot\!r_1^{2n_1-1}\cdot(1+n_2)^{m_1-1}\\\vspace{0.2cm}
&\displaystyle\times 2^{m_1-n_1-1}\cdot\frac{n_3r_1r_3+4n_2r_1+2n_2r_2+2n_3r_1+4r_1}{2m_1+n_1m_2+m_1m_3+2m_1n_2+m_1n_3}.\\
 \end{array}
\end{equation*}
\begin{equation*}
 \begin{array}{rl}
 \hspace{-1cm}(2)~~\tau(\mathfrak{G})&=2^{2m_1-n_1-1}\cdot\!\prod\limits_{i=1}^{n_1-1}(r_1+r_1n_2)\theta_i\cdot\!\prod\limits_{j=1}^{n_2-1}\!(2\!+\!r_2\mu_j)^{m_1}
 \cdot\prod\limits_{k=1}^{n_3-1}(1\!+\!r_3\eta_k)^{n_1}\\
 &\displaystyle\times\frac{n_2r_1r_2+4n_2r_1+2n_3r_3+4n_3+4r_1}{2m_1+m_1m_2+n_1m_3+n_1n_3+2m_1n_2}.
\end{array}
\end{equation*}
\end{thm}
\begin{proof}
The proof of (2) is similar to the proof of (1), it is here need to prove (1). We first consider the normalized Laplacian eigenvalues of $\mathcal{G}$ in the following way:

In Corollary \ref{cor-3-1} (c), one can by the well-known \emph{Vieta Theorem} obtain the relation of the two roots $\alpha_{1}$ and $\alpha_2$ of Eq.(\ref{eq-3.1-c}) such that
\begin{equation}\label{eq-4.2-1}
\alpha_1\alpha_2=\frac{2n_2+2}{(r_3+1)(2n_2+n_3+2)}.
\end{equation}
In Corollary \ref{cor-3-1} (d), let $\beta_1,\beta_2,\beta_3$ and $\beta_4$ be the four roots of Eq.(\ref{eq-3.1-d}) for each $\theta_{i},i=1,2,\ldots,n_{1}-1$. Then
\begin{equation}\label{eq-4.2-2}
\beta_1\beta_2\beta_3\beta_4=\frac{r_1(1+n_2)\theta_i}{(r_1+r_2)(1+r_3)(2n_2+n_3+2)}.
\end{equation}
For $i=n_{1}$, we notice that $\theta_{n_{1}}=0$. By Eq.(\ref{eq-3.1-d}) one can get
\begin{equation}\label{eq-4.2-3}
\begin{array}{rl}
&(2n_2r_1r_3+2n_2r_2r_3+n_3r_1r_3+n_3r_2r_3+2n_2r_1+2n_2r_2
+n_3r_1+n_3r_2+2r_1r_3\\
&+2r_2r_3+2r_1+2r_2)\lambda^4-(6n_2r_1r_3+4n_2r_2r_3+3n_3r_1r_3+2n_3r_2r_3
+8n_2r_1\\
&+6n_2r_2+4n_3r_1+3n_3r_2+6r_1r_3+4r_2r_3+8r_1+6r_2)\lambda^3+(4n_2r_1r_3+2n_2r_2r_3\\
&+3n_3r_1r_3+n_3r_2r_3+10n_2r_1+6n_2r_2+5n_3r_1+2n_3r_2+4r_1r_3+10r_1+4r_2)\lambda^2\\
&-(n_3r_1r_3+4n_2r_1+2n_2r_2+2n_3r_1+4r_1)\lambda=0.
\end{array}
\end{equation}
Suppose that $\gamma_1,\gamma_2$ and $\gamma_3$ are three non-zero roots of Eq.(\ref{eq-4.2-3}). Then by \emph{Vieta Theorem},
\begin{equation}\label{eq-4.2-4}
\gamma_1\gamma_2\gamma_3=\frac{n_3r_1r_3+4n_2r_1+2n_2r_2+2n_3r_1+4r_1}{(r_1+r_2)(1+r_3)(2n_2+n_3+2)}.
\end{equation}
In light of Corollary \ref{cor-3-1}, Eqs.(\ref{eq-4.2-1}), (\ref{eq-4.2-2}) and (\ref{eq-4.2-4}) we see that
\begin{equation*}
\begin{array}{rl}
\tau(\mathcal{G})\!\!\!\!\!&=\displaystyle\frac{\prod_{i=1}^{n}d_i\prod_{i=1}^{n-1}\lambda_i}{\sum_{i=1}^{n}d_i}\\
&=\frac{(r_1)^{n_1}(2n_2+n_3+2)^{m_1}(r_1+r_2)^{n_1n_2}(r_3+1)^{m_1n_3}}{2(2m_1+n_1m_2+m_1m_3+2m_1n_2+m_1n_3)}\prod\limits_{k=1}^{n_3-1}\!(\frac{1+r_3\eta_k}{r_3+1})^{m_1}\!\!\cdot\!\prod\limits_{i=1}^{n_1-1}\!\frac{r_1(1+n_2)\theta_i}{(r_1\!+\!r_2)(1\!+\!r_3)(2n_2\!+\!n_3\!+\!2)}\\
&\times\prod\limits_{j=1}^{n_2-1}\!(\frac{r_1+r_2\mu_j}{r_1+r_2})^{n_1}\!\!\cdot\!(\frac{2n_2+2}{(r_3+1)(2n_2\!+\!n_3\!+\!2)})^{m_1\!-\!n_1}\!\!\cdot\!\frac{n_3r_1r_3+4n_2r_1+2n_2r_2+2n_3r_1+4r_1}{(r_1+r_2)(r_3+1)(2n_2+n_3+2)}\\\vspace{0.2cm}
&=\displaystyle\prod\limits_{i=1}^{n_1-1}\!\theta_i\cdot\!\prod\limits_{j=1}^{n_2-1}\!(r_1\!+\!r_2\mu_j)^{n_1}\!\cdot\!\prod\limits_{k=1}^{n_3-1}\!(1\!+\!r_3\eta_k)^{m_1}\!\cdot\!r_1^{2n_1-1}\cdot(1+n_2)^{m_1-1}\cdot2^{m_1-n_1-1}\\\vspace{0.2cm}
&\displaystyle\times\frac{n_3r_1r_3+4n_2r_1+2n_2r_2+2n_3r_1+4r_1}{2m_1+n_1m_2+m_1m_3+2m_1n_2+m_1n_3},
\end{array}
\end{equation*}
as required.
\end{proof}

\begin{example}\label{examp-4.1-2}
Let $G=C_4^S\bowtie(K_2^V\cup K_2^E)$ and $H=C_4^S\diamondsuit(K_2^V\cup
K_1^E)$ (shown in Fig.\ref{fig-3.2}). It is easy to see that
$\prod_{i=1}^{n_1-1}\!\theta_i=2$, $\prod_{j=1}^{n_2-1}(r_1\!+\!r_2\mu_j)^{n_1}=4^4=2^8$,
$\prod_{k=1}^{n_3-1}(r_1\!+\!r_3\eta_k)^{m_1}\cdot r_{1}^{2n_1-1}=3^4\cdot 2^{7}$, $(1+n_2)^{m_1-1}\cdot2^{m_1-n_1-1}=3^3\cdot 2^{-1}$,
$n_3r_1r_3+4n_2r_1+2n_2r_2+2n_3r_1+4r_1=40$, $2m_1+n_1m_2+m_1m_3+2m_1n_2+m_1n_3=40$. Thus, by Theorem \ref{thm-4.2-1} (1) we get
$\tau(G)=2^{15}\cdot 3^{7}$. On the other hand, combining to the Example \ref{examp-3-1} and Eq.(\ref{eq-4.1-1}), one can easy to obtain that $\tau(G)=2^{15}\cdot 3^{7}$. Similarly, $\tau(H)=2^{15}\cdot 3^{3}$.
\end{example}

\subsection{The multiplicative degree-Kirchhoff index}
In \cite{Chen}, the \emph{multiplicative degree-Kirchhoff index} of $G$ is defined as
$$Kf^{*}(G)=\sum\limits_{i<j}d_{i}d_{j}r_{ij}$$
by Chen and Zhang, where $r_{ij}$ is the resistance between $i$ and $j$. This index is distinct the classical \emph{Kirchhoff index} $Kf(G)=\sum_{i<j}r_{ij}$ since it takes into account the degree distribution of $G$.
Meanwhile, they also have been proved that $Kf^{*}(G)$ can be obtained from the non-zero normalized Laplacian eigenvalues of $G$, i.e.,
\begin{equation}\label{eq-4.3-1}
\hspace{4.5cm}Kf^{*}(G)=2m\cdot\sum\limits_{i=1}^{n-1}\frac{1}{\lambda_{i}}.
\end{equation}

\begin{thm}\label{thm-4.3-1}
Let $G_i$ be an $r_i$-regular graph with $n_i$ vertices and $m_i$ edges for $i=1,2,3$. Then
\begin{equation*}
\begin{array}{rl}\vspace{0.2cm}
(1)~~K\!f^{*}(\mathcal{G})\!\!&=2(2m_1+n_1m_2+m_1m_3+2m_1n_2+m_1n_3)\big(\sum\limits_{j=1}^{n_2-1}\frac{n_1(r_1+r_2)}{r_1+r_2\mu_j}+\sum\limits_{k=1}^{n_3-1}\frac{m_1(r_3+1)}{1+r_3\eta_k}\\
&+\frac{(m_1\!-\!n_1)(r_3\!+\!2)(2n_2\!+\!n_3\!+\!2)}{2n_2\!+\!2}\!+\!\frac{(2n_2+n_3+2)(6r_1+3r_2+3r_1r_3+r_2r_3)-n_3(r_1+r_2)\!-\!2(r_3+1)(r_1+r_2+r_1n_2)}{(4r_1+r_2+r_1r_3)(2n_2+n_3+2)\!-\!(2r_1n_3+r_2n_3+2r_2)-(2+2n_2)(2r_1+r_1r_3)}\\
&+\sum\limits_{i=1}^{n_1-1}\frac{(4r_1+r_2+r_1r_3)(2n_2+n_3+2)-n_3(2r_1+r_2)-(2-\theta_i)\big((1+n_2)(2r_1+r_1r_3)+r_2\big)}{r_1(1+n_2)\theta_i}\big).
\end{array}
\end{equation*}
\begin{equation*}
\begin{array}{rl}\vspace{0.2cm}
\hspace{-0.85cm}(2)~~K\!f^*(\mathfrak{G})\!\!&=2(2m_1+m_1m_2+n_1m_3+n_1n_3+2m_1n_2)\big(\sum\limits_{j=1}^{n_2-1}\frac{m_1(r_2+2)}{2+r_2\mu_j}+\sum\limits_{k=1}^{n_3-1}\frac{n_1(r_3+1)}{1+r_3\eta_k}\\
&+\frac{(r_2+4)(m_1-n_1)}{2}+\frac{2(r_1n_2+r_1+n_3)(r_2r_3+6r_3+3r_2+12)-n_3(2r_2+4)-2r_1(r_3+1)(2n_2+r_2+2)}{2n_2r_1r_2+8n_2r_1+4n_3r_3+8n_3+8r_1}\\
&+\sum\limits_{i=1}^{n_1-1}\frac{2(r_1n_2+r_1+n_3)(2r_3+r_2+8)-(2r_2n_3+8n_3)-r_1(2-\theta_i)((2+2n_2)(2+r_3)+r_2)}{r_1(2+2n_2)\theta_i}\big).
\end{array}
\end{equation*}
\end{thm}
\begin{proof}
 From Eq.(\ref{eq-4.3-1}), $Kf^*(\mathcal{G})$ can be computed from the following way:

In Corollary \ref{cor-3-1} (c), let $\alpha_1$ and $\alpha_2$ be the two eigenvalues of equation (\ref{eq-3.1-c}). Then by \emph{Vieta Theorem}, we have
$$\frac{1}{\alpha_1}+\frac{1}{\alpha_2}=\frac{\alpha_1+\alpha_2}{\alpha_1\alpha_2}=\frac{(r_3+2)(2n_2+n_3+2)}{2n_2+2}.$$
In Corollary \ref{cor-3-1} (d), for each $\theta_{i}(i=2,3,\ldots,n_1)$, let $\beta_1$, $\beta_2$, $\beta_3$ and $\beta_4$ are the eigenvalues of Eq.(\ref{eq-3.1-d}). By \emph{Vieta Theorem}, we have
\begin{equation*}
\begin{array}{rl}\vspace{0.2cm}
\frac{1}{\beta_1}+\frac{1}{\beta_2}+\frac{1}{\beta_3}+\frac{1}{\beta_4}&=\frac{\beta_2\beta_3\beta_4+\beta_1\beta_3\beta_4+\beta_1\beta_2\beta_4+\beta_1\beta_2\beta_3}{\beta_1\beta_2\beta_3\beta_4}\\ &=\frac{(4r_1+r_2+r_1r_3)(2n_2+n_3+2)-n_3(2r_1+r_2)-(2-\theta_i)\big((1+n_2)(2r_1+r_1r_3)+r_2\big)}{r_1(1+n_2)\theta_i}.
\end{array}
\end{equation*}

Note that $\theta_{n_{1}}=0$. Then Eq.(\ref{eq-3.1-d}) equals to Eq.(\ref{eq-4.2-3}). Let $\gamma_1$, $\gamma_2$ and $\gamma_3$ be the non-zero eigenvalues of Eq.(\ref{eq-4.2-3}). Then
\begin{equation*}
\begin{array}{rl}\vspace{0.2cm}
\frac{1}{\gamma_1}+\frac{1}{\gamma_2}+\frac{1}{\gamma_3}&=\frac{\gamma_2\gamma_3+\gamma_1\gamma_3+\gamma_1\gamma_2}{\gamma_1\gamma_2\gamma_3}\\ &=\frac{(2n_2+n_3+2)(6r_1+3r_2+3r_1r_3+r_2r_3)-n_3(r_1+r_2)-2(r_3+1)(r_1+r_2+r_1n_2)}{(4r_1+r_2+r_1r_3)(2n_2+n_3+2)-(2r_1n_3+r_2n_3+2r_2)-(2+2n_2)(2r_1+r_1r_3)}\\
\end{array}
\end{equation*}

In summary above, the results of (1) follows. Similarly, (2) can be obtained.
\end{proof}
\begin{example}\label{examp-4.1-3}
For the graphs $G=C_4^S\bowtie(K_2^V\cup K_2^E)$ and $H=C_4^S\diamondsuit(K_2^V\cup
K_1^E)$ (shown in Fig.\ref{fig-3.2}), in light of Theorem \ref{thm-4.3-1}, $Kf^{*}(G)=\frac{2123\times 80}{60}=\frac{8492}{3}$. On the other hand, combining to the Example \ref{examp-3-1} and Eq.(\ref{eq-4.3-1}), one can easy to obtain that $Kf^{*}(G)=\frac{8492}{3}$. Similarly, $Kf^{*}(H)=\frac{307\times 64}{12}=\frac{4912}{3}$.
\end{example}
\subsection{Kemeny's constant}
For a graph $G$ , \emph{Kemeny's constant} $K(G)$, also known as average hitting time, is the expected number of steps required for the transition from a starting vertex $i$ to a destination vertex, which is chosen randomly according to a stationary distribution of unbiased random walks on $G$, see \cite{Hunter} for more details. From literature \cite{Butler} we know that
\begin{equation*}
\hspace{4.8cm} K(G)=\sum\limits_{i=1}^{n-1}\frac{1}{\lambda_{i}}.
\end{equation*}
Note that $K\!f^*(G)=2m\cdot K(G)$. Thus, the following result follows from Theorem \ref{thm-4.3-1} immediately.
\begin{thm}\label{thm-4.4-1}
Let $G_i$ be an $r_i$-regular graph with $n_i$ vertices and $m_i$ edges, where $i=1,2,3$. Then
\begin{equation*}
\begin{array}{rl}\vspace{0.2cm}
(1)~~K(\mathcal{G})\!\!&=\sum\limits_{j=1}^{n_2-1}\frac{n_1(r_1+r_2)}{r_1+r_2\mu_j}+\sum\limits_{i=1}^{n_1-1}\frac{(4r_1+r_2+r_1r_3)(2n_2+n_3+2)-n_3(2r_1+r_2)-(2-\theta_i)\big((1+n_2)(2r_1+r_1r_3)+r_2\big)}{r_1(1+n_2)\theta_i}\\
&+\frac{(m_1\!-\!n_1)(r_3\!+\!2)(2n_2\!+\!n_3\!+\!2)}{2n_2\!+\!2}\!+\!\frac{(2n_2+n_3+2)(6r_1+3r_2+3r_1r_3+r_2r_3)-n_3(r_1+r_2)\!-\!2(r_3+1)(r_1+r_2+r_1n_2)}{(4r_1+r_2+r_1r_3)(2n_2+n_3+2)\!-\!(2r_1n_3+r_2n_3+2r_2)-(2+2n_2)(2r_1+r_1r_3)}\\
&+\sum\limits_{k=1}^{n_3-1}\frac{m_1(r_3+1)}{1+r_3\eta_k}.\\
(2)~~K(\mathfrak{G})\!\!&=\sum\limits_{i=1}^{n_1-1}\frac{2(r_1n_2+r_1+n_3)(2r_3+r_2+8)-(2r_2n_3+8n_3)-r_1(2-\theta_i)((2+2n_2)(2+r_3)+r_2)}{r_1(2+2n_2)\theta_i}+\sum\limits_{j=1}^{n_2-1}\frac{m_1(r_2+2)}{2+r_2\mu_j}\\
&+\frac{(r_2+4)(m_1-n_1)}{2}+\frac{2(r_1n_2+r_1+n_3)(r_2r_3+6r_3+3r_2+12)-n_3(2r_2+4)-2r_1(r_3+1)(2n_2+r_2+2)}{2n_2r_1r_2+8n_2r_1+4n_3r_3+8n_3+8r_1}\\
&+\sum\limits_{k=1}^{n_3-1}\frac{n_1(r_3+1)}{1+r_3\eta_k}.
\end{array}
\end{equation*}
\end{thm}
\begin{example}\label{examp-4.1-4}
For graphs $G=C_4^S\bowtie(K_2^V\cup K_2^E)$ and $H=C_4^S\diamondsuit(K_2^V\cup
K_1^E)$ (shown in Fig.\ref{fig-3.2}), according to Theorem \ref{thm-4.4-1}, one can get $K(G)=\frac{2123}{60}$, $K(H)=\frac{307}{12}$.
\end{example}


\begin{thebibliography}{20}
{\scriptsize

\bibitem{Chung} F.R.K. Chung, Spectral graph theory. CBMS. Regional conference series in mathematics. Vol.92, Providence (RI): AMS; 1997.\vspace{-0.25cm}

\bibitem{1} D. Cvetkovi\'{c}, P. Rowlinson, S.K.
Simi\'{c}, An Introduction to the Theory of Graph
 Spectra [M], Cambridge University Press, Cambridge, 2010.\vspace{-0.33cm}

\bibitem{4} S.Y. Cui, G.X. Tian, The spectrum and the signless Laplacian spectrum of corona, Linear Algabra Appl. 437 (2012) 1692--2703.\vspace{-0.25cm}

\bibitem{6} I. Gopalapillai, The spectrum of neighborhood corona of graphs, Kragujevac J. Math. 35(2011):493--500.\vspace{-0.25cm}

\bibitem{7} Y.P. Hou, W.C. Shiu, The spectrum of edge corona two graphs, Electron. J. Linear Algebra. 20(2010):586--594.\vspace{-0.25cm}

\bibitem{11} C. McLeman, E. McNicholas, Spectra of coronae, Linear Algebra Appl. 435(2011):998--1007.\vspace{-0.25cm}

\bibitem{12} S.L. Wang, B. Zhou, The signless Laplacian spectra of  corona and edge corona of two graphs, Linear and Multilinear Algebra. (2012) 1--8, iFirst.\vspace{-0.25cm}

\bibitem{8} G. Indulal, Spectrum of two new joins of graphs and infinite famillies of integral graphs, Kragujevac J. Math. 36 (2012) 133--139.\vspace{-0.25cm}

\bibitem{9} X.G. Liu, Z.H. Zhang, Spectra of subdivision-vertex and subdivision-edge joins of graphs, submitted for publication, ArXiv:1212.0619.\vspace{-0.25cm}

\bibitem{13} X.G. Liu, P.L. Lu, The spectra of the subdivision-vertex and subdivision-edge neighborhood corona, Linear Algebra Appl. 438(2013):3547--3559.\vspace{-0.25cm}

\bibitem{10} P.L. Lu, Y.F. Miao, Spectra of subdivision-vexter and subdivision-edge corona, in preparation.\vspace{-0.25cm}

\bibitem{song} C.X. Song, Q.X. Huang, Spectra of subdivision
vertex-edge coronae for graphs, Advances in Mathematics (China), 2016, 45(1):38--47.\vspace{-0.25cm}


\bibitem{zhang} F.Z. Zhang, The Schur Complement and its Application, springer, 2005.\vspace{-0.25cm}


\bibitem{xie} P. Xie, Z. Zhang , F. Comellas, The normalized Laplacian spectrum of subdivisions of a graph, Appl. Math. Comput., 2016, 286(C): 250--256. \vspace{-0.25cm}

\bibitem{Yu} P.K. Yu, G.X. Tian, The normalized Laplacian spectra of the double corona based on $R$-graph, 2017, arxiv.org/ abs/1709.02687v1.\vspace{-0.25cm}

\bibitem{Das} A. Das,  P. Panigrahi, Normalized Laplacian spectrum of some subdivision-coronas of two regular graphs, Linear Multilinear Algebra, 2017, 65(5): 962--972.\vspace{-0.25cm}

\bibitem{ER} E.R. van Dam and W.H. Haemers, Which graphs are
determined by their spectrum? Linear Algebra Appl., 373(2003):
241--272.\vspace{-0.25cm}

\bibitem{Chen} H. Chen, F. Zhang, Resistance distance and the normalized Laplacian spectrum, Disc. Appl. Math. 155(2007): 654--661.\vspace{-0.25cm}

\bibitem{Hunter} J.J. Hunter, The role of Kemeny's constant in properties of Markov chains, Commun. Stat. Theor. Methods, 43(2014):1309--1321.\vspace{-0.25cm}

\bibitem{Butler} S. Butler, Algebraic aspects of the normalized Laplacian, Recent Trends in Combinatorics, The IMA Volumes in Mathematics and its Applications, IMA, 2016. To appear. \vspace{-0.25cm}

\vspace{-0.33cm}
}
\end{thebibliography}
\end{document}